\documentclass[11pt,a4paper]{amsart}

\usepackage{fullpage}

\usepackage{graphicx} 
\usepackage[utf8]{inputenc}
\usepackage[english]{babel}
\usepackage[T1]{fontenc}
\usepackage{amssymb}
\usepackage{mathtools}
\usepackage{amsfonts}
\usepackage{amsmath}
\usepackage{amsthm}
\usepackage{mathrsfs}
\usepackage{tikz}
\usepackage{tikz-cd}
\usepackage{changepage}
\graphicspath{{images/}}
\usepackage{fancyhdr}
\usepackage[english]{babel}
\usepackage{amsmath}
\usepackage{amssymb}
\usepackage{amsfonts}
\usepackage{amsthm}
\usepackage{csquotes}

\usepackage[backend=biber,style=numeric, maxbibnames=50]{biblatex}
\usepackage[hyperindex,breaklinks]{hyperref}
\hypersetup{
	colorlinks=true,
	linkcolor=blue,
	citecolor=blue,
	urlcolor=blue,
	pdfauthor={Mohammad Alattar},
	pdftitle={}
}

\usepackage{xurl}

\hypersetup{breaklinks=true}

\DeclareMathOperator{\Isom}{Isom}

\newtheorem{theorem}{Theorem}  
\newtheorem{cor}[theorem]{Corollary}  

\theoremstyle{plain}
\newtheorem{thm}{Theorem}
\newtheorem{lemma}[thm]{Lemma}
\newtheorem{corollary}[thm]{Corollary}

\newtheorem{prop}[thm]{Proposition}

\theoremstyle{definition}
\newtheorem{definition}[thm]{Definition}

\newtheorem{remark}[thm]{Remark}

\numberwithin{equation}{section}
\numberwithin{thm}{section}

\usepackage{setspace}

\theoremstyle{remark} 
\newtheorem*{ack}{Acknowledgements}

\begin{document}

\title{Stability and equivariant Gromov--Hausdorff convergence}

\author[Mohammad Al-Attar]{Mohammad Alattar}
\address[Al-Attar]{Department of Mathematical Sciences, Durham University, United Kingdom}
\email{\href{mailto:mohammad.al-attar@durham.ac.uk}{mohammad.al-attar@durham.ac.uk}}

\date{\today}

\begin{abstract}
   We give applications of equivariant Gromov--Hausdorff convergence in various contexts. Namely, using equivariant Gromov--Hausdorff convergence, we prove a stability result in the setting of compact finite dimensional Alexandrov spaces. Moreover, we introduce the notion of an \emph{almost commutative diagram} and show that its use simplifies both exposition and argument. 
    
    \end{abstract}


 
\maketitle

\section{Introduction}

In the early 1990s, Burago, Gromov and Perelman developed the theory of Alexandrov spaces \cite{burago1992ad}. These spaces are metric space generalizations of complete Riemannian manifolds with a lower sectional curvature bound. Typically, in dimensions greater than two, Alexandrov spaces are not manifolds. In general, they are almost manifolds, in the sense that they contain an open dense subset that is a manifold \cite{burago2022course}. Therefore, one cannot automatically transfer the techniques of smooth manifold theory to the theory of Alexandrov spaces. In 1991, Perelman established a stability theorem \cite{Vitali,perelman1991alexandrov}, asserting that if a sequence of compact $n$-dimensional Alexandrov spaces $\{X_k\}_{k\in \mathbb{N}}$ with a uniform lower curvature bound, Gromov--Hausdorff converges to another compact Alexandrov space $X$ with no collapse, then $X_k$ and $X$ are homeomorphic for all sufficiently large $k$.

   In the 1980s, Fukaya introduced a pointed equivariant version of Gromov--Hausdorff convergence \cite{fukaya1986theory}. 
   In 1994, Fukaya and Yamaguchi used the equivariant analogue of Gromov--Hausdorff convergence to establish that the isometry group of an Alexandrov space is a Lie group with respect to the compact-open topology \cite{fukaya1994isometry}. The proof shows that the equivariant version of Gromov--Hausdorff convergence is a powerful tool. In addition, the equivariant Gromov--Hausdorff topology has been applied to the framework of $\mathrm{RCD}^{*}(K,N)$ spaces. In particular, Guijarro and Santos-Rodriguez \cite{GR} and, independently, Sosa \cite{Sosa}, proved using equivariant Gromov--Hausdorff convergence that the isometry group of an $\mathrm{RCD}^{*}(K,N)$ space is a Lie group. These results for Alexandrov and $\mathrm{RCD}^*$ spaces are the departure points for further developments in the study of isometries of singular spaces, which has attracted some  attention (see, for example, \cite{Corro.et.al,Garcia-Guijarro,GGSearle,Garcia2,Garcia3,GG2,harvey2016convergence,harvey2016equivariant,HarveySearle,nunez}).

Our first main theorem is the following stability result with respect to the equivariant Gromov--Hausdorff convergence, generalizing the work of Harvey (\cite{harvey2016equivariant}, Theorem 3.6). Fix a natural number $n$. Denote by $\mathcal{A}(n,K,D)$ to be the class of compact $n$-dimensional Alexandrov spaces with curvature bounded below by $K$ and diameter bounded above by $D$.


\begin{theorem}[Generalized Equivariant Stability]
\label{T:theorem.A}
Let $\{X_k\}_{k=1}^{\infty}$ and $X$ be in $\mathcal{A}(n,K,D)$.  Fix a natural number $m$. Assume that $G_k$ is an $m$-dimensional closed subgroup of $\Isom(X_k)$. Assume $(X_k,G_k)\rightarrow_{eGH}(X,G)$ and that $G$ is an $m$-dimensional closed subgroup of $\Isom(X)$. Then for all large enough $k$, there exists a Lie group isomorphism $\theta_k\colon G_k\rightarrow G$ and a homeomorphism $\xi_k\colon X_k\rightarrow X$ such that if $g_k\in G_k$ then $\xi_k\circ g_k=\theta_k(g_k)\circ \xi_k$.
\end{theorem}

Note that by definition of equivariant Gromov--Hausdorff convergence  (\ref{def}), the actions of $G_k$ and $G$ on $X_k$ and $X$ respectively are \textit{effective}. This property will be crucial in the proof (see \ref{effective remark}). Further note that by \cite{fukaya1994isometry}, each $G_k$ and $G$ is a compact Lie group.

The preceding theorem generalizes Harvey's stability theorem, allowing for different groups acting on the Alexandrov spaces under consideration.


As a corollary to Theorem \ref{T:theorem.A}, we obtain the following result. We denote by $EG$ the universal space. Recall that $EG$ is the countably infinite join of $G$ and $G$ acts freely on $EG$ \cite{Tu}.


\begin{cor}
\label{C:corollary.C}
Let $\{X_k\}_{k=1}^{\infty}$ and $X$ be in $\mathcal{A}(n,K,D)$.  Fix a natural number $m$. Assume that $G_k$ is an $m$-dimensional closed subgroup of $\Isom(X_k)$. Assume $(X_k,G_k)\rightarrow_{eGH}(X,G)$ and that $G$ is an $m$-dimensional closed subgroup of $\Isom(X)$. Then for all sufficiently large $k$, the Borel construction $EG_k\times_{G_k} X_k$ is homeomorphic to the Borel construction $EG\times_{G} X$.
\end{cor}

Our next main theorem is the following stability result concerning the continuity of maps under equivariant Gromov--Hausdorff convergence. 

\begin{theorem}
\label{T:theorem.B}
Let  $\{X_k\}_{k\in \mathbb{N}}$ be a sequence of compact metric spaces. For each $k$, let $G_k$ denote a closed subgroup of the isometry group of $X_k$. Assume that a triple $(f_k,\theta_k,\psi_k)$ witnesses the convergence $(X_k,G_k)\rightarrow_{eGH}(X,G)$, where $X$ is a compact metric space, and $G$ is a closed subgroup of the isometry group of $X$. If $G$ is a euclidean neighborhood retract (ENR), then the maps $\theta_k\colon G_k\rightarrow G$ may be taken to be continuous.
\end{theorem}

For a definition of a \emph{triple that witnesses the convergence}, we refer the reader to definition~\ref{def} (cf. \cite{harvey2016convergence,harvey2016equivariant}). For the definition of an ENR, we refer the reader to \cite{Hatcher}. 

Theorem~\ref{T:theorem.B} is a generalization of one of the technical results used to prove the equivariant stability theorems in \cite{harvey2016convergence,harvey2016equivariant}. The stability result in \cite{harvey2016convergence} is proved for pointed convergence, proper spaces, and when all the groups $G_k$ and $G$ are Lie groups. In \cite{harvey2016equivariant}, the stability result is proved for compact spaces and when $G_k=G$ for all $k$ and $G$ is a compact Lie group. In both proofs, Harvey used the center of mass construction by Grove--Petersen \cite{Grove-Petersen} to approximate discrete maps by continuous maps. In Theorem~\ref{T:theorem.B}, we do not assume that the groups $G_k$ are Lie groups.  Our proof is elementary and does not require Lie group theory nor the center of mass construction. Instead, we only use the definition and properties of equivariant Gromov--Hausdorff convergence and the definition of a euclidean neighborhood retract (ENR).
\\

Our article is organized as follows. In section  \ref{s:equivariant.gh.convergence}, we define the notion of equivariant Gromov--Hausdorff convergence and prove elementary properties. In section \ref{s:proof.technical.thm}, we prove Theorem~\ref{T:theorem.B}. In section \ref{s:proof.stability.thm}, using some ideas from section \ref{s:proof.technical.thm}, we prove Theorem~\ref{T:theorem.A} and Corollary \ref{C:corollary.C}.

\begin{ack}
The author would like to thank Phatimah Alhazmi, Gabriel Arenas Henriquez, Mauricio Che, Saghar Hosseinisemnani, Alexander Jackson, and Viktor Matyas for useful discussions while preparing this manuscript. The author would especially like to thank his advisor Fernando Galaz--García, Jaime Santos-Rodríguez, Martin Kerin, and Kohei Suzuki.
\end{ack}


\section{Equivariant Gromov--Hausdorff Convergence}
\label{s:equivariant.gh.convergence}

In this section, we introduce the equivariant Gromov--Hausdorff distance as initially defined by Fukaya \cite{fukaya1986theory,fukayaandyamaguchiannals}. However, we assume our spaces are compact and as such, we do not require the use of basepoints. We denote the isometry group of a metric space $X$ by $\Isom(X$). Once $X$ is compact, the topology on the space of continuous functions on $X$ is metrizable, and we can take the metric to be the uniform metric.  Due to the technical nature of equivariant Gromov--Hausdorff approximations, and for the sake of easing exposition, we introduce the notion of an \emph{almost commutative diagram}.


\begin{definition}
Let $(X,d_{X}),(Y,d_{Y})$ be metric spaces and fix $\epsilon\geq 0$. We say that the diagram

\begingroup
\centering

\begin{tikzcd}
X \arrow[r, "f"] \arrow[d, "k"] & Y \arrow[d, "g"] \\
X \arrow[r, "h"]                & Y               
\end{tikzcd}

\endgroup

\noindent \textit{commutes up to $\epsilon$} (\textit{$\epsilon$--commutes}) if $d_{Y}(gf(x),hk(x))\leq \epsilon$ for all $x\in X$.  In this case, we will usually write \emph{$g\circ f = h\circ k$ up to $\epsilon$}.
\end{definition}


\begin{remark} When $\epsilon=0$ in the diagram above, we recover the notion of a commutative diagram. Hence, we may think of $\epsilon-$commutative diagrams as being almost commutative.
\end{remark}


\begin{definition}
\label{defeGH}
Let $X$ and $Y$ be compact metric spaces and let $G_{X}$ and $G_{Y}$ be closed subgroups of Isom($X$) and Isom($Y$) respectively. An \textit{equivariant Gromov--Hausdorff approximation of order $\epsilon>0$} between the pairs $(X,G_{X})$ and $(Y,G_{Y})$ is a triple of maps
\[ 
(f\colon X\rightarrow Y,\theta\colon G_{X}\rightarrow G_{Y},\psi\colon G_{Y}\rightarrow G_{X}) 
\]
subject to the following conditions:

\begin{enumerate}
    \item The map $f$ is a Gromov--Hausdorff approximation of order $\epsilon$ (i.e., an $\epsilon$-isometry).
    \item If $\gamma \in G_{X}$ then $\theta(\gamma)\circ f = f\circ \gamma$ up to $\epsilon$.
    \item If $\lambda \in G_{Y}$ then $\lambda \circ f =f\circ \psi(\lambda)$ up to $\epsilon$.

\end{enumerate}

Often, when the spaces $X, Y$ and the groups $G_{X}$ and $G_{Y}$ are understood from context, we refer to the triple $(f,\theta,\psi)$ as an \emph{equivariant Gromov--Hausdorff approximation}.
\end{definition}

Our definition of an equivariant Gromov--Hausdorff approximation is, to the best of our knowledge, not found in the literature. It is, however, trivially equivalent to the standard one (see, for example, \cite{fukaya1986theory}).

Before continuing, we note a few things. First, we use the terms \emph{almost isometries}, \emph{Gromov--Hausdorff approximations}, and \emph{approximations} interchangeably. Similarly, we use the terms \textit{equivariant Gromov--Hausdorff approximations} and \textit{equivariant approximations} interchangeably. Further, an \textit{$\epsilon$-isometry}, where $\epsilon>0$, is the same as an \textit{almost isometry of order $\epsilon>0$}. Similar conventions will be in place for the (almost) equivariant analogue.

Second, we note that the maps $\theta$ and $\psi$ in definition \ref{defeGH} need not be group homomorphisms and $f$ need not be continuous. Recall that an almost isometry of order $\epsilon$, between compact metric spaces $f\colon X\rightarrow Y$ admits an almost inverse $\tilde{f}\colon Y\rightarrow X$ that is a $3\epsilon$-isometry \cite{jansen2017notes}. In particular, $f\circ \tilde{f}=\mathrm{id}$  up to $\epsilon$  and $\tilde{f}\circ f=\mathrm{id}$ up to $\epsilon$. Thus, it is natural to wonder, given an $\epsilon$-equivariant Gromov--Hausdorff approximation $(f,\theta,\psi)$ between $(X,G_{X})$ and $(Y,G_{Y})$, whether there exists an equivariant Gromov--Hausdorff approximation $(\tilde{f},\tilde{\theta},\tilde{\psi})$ between $(Y,G_{Y})$ and $(X,G_{X})$ that serves as an almost inverse to $(f,\theta,\psi)$. The answer is in the affirmative, as we shall show. First, let us sketch the proof in the non-equivariant setting. We will then state and prove the equivariant version.


\begin{lemma}[\cite{jansen2017notes}]
\label{lem:jansen}
Let $f\colon X\rightarrow Y$ be an $\epsilon$-isometry between compact metric spaces. Then there exists a $3\epsilon$-isometry $\tilde{f}\colon Y\rightarrow X$ such that $d_{Y}(f(\tilde{f}(y)),y)\leq 3 \epsilon$ for all $y\in Y$ and $d_{X}(\tilde{f}(f(x)),x)\leq 3\epsilon$ for all $x\in X$.
\end{lemma}

\begin{proof}[Proof (Sketch)]
 For each $y\in Y$, choose an $x_y\in X$ such that $d_{Y}(f(x_{y}),y)\leq \epsilon$. Then define $\tilde{f}\colon Y\rightarrow X$ given by $\tilde{f}(y)=x_y$. 
\end{proof}

\begin{prop}
\label{prop:4e.equiv.GH.approximation}
Let $X$ and $Y$ be compact metric spaces. Let $G_{X}$ and $G_{Y}$ be closed subgroups of $\Isom(X)$ and $\Isom(Y)$ respectively. Let $(f,\theta,\psi)$ be an $\epsilon$-equivariant Gromov--Hausdorff approximation between $(X,G_{X})$ and $(Y,G_{Y})$. Then there there exists a $4\epsilon$-equivariant Gromov--Hausdorff approximation between $(Y,G_{Y})$ and $(X,G_{X})$.
\end{prop}

\begin{proof}
As $f$ is an $\epsilon$-Gromov--Hausdorff approximation, for each $y\in Y$, we may find  $x_y\in X$ such that $d_{Y}(f(x_y),y)\leq \epsilon$. Now define $\tilde{f}\colon Y\rightarrow X$ by $\tilde{f}(y)=x_y$. Lemma~\ref{lem:jansen} shows that $\tilde{f}$ is an $3\epsilon$- isometry (and thus a $4\epsilon$-isometry). Now we show that $(\tilde{f},\psi,\theta)$ is a $4\epsilon$-equivariant Gromov--Hausdorff approximation between $(Y,G_{Y})$ and $(X,G_{X})$.

Let $\lambda \in G_{Y}$. We will show that $\psi(\lambda)\circ \tilde{f}= \tilde{f}\circ \lambda$ up to $4\epsilon$. Let $y\in Y$. Then, $\tilde{f}(y)=x_y$, where $x_y$ is chosen such that $d_{Y}(f(x_y),y)\leq \epsilon$. Similarly, for $\lambda y\in Y$,  $\tilde{f}(\lambda y)=x_{\lambda y}$, where $x_{\lambda y}$ is chosen so that $d_{Y}(f(x_{\lambda y}),\lambda y)\leq \epsilon$. Since $f$ is an $\epsilon$-isometry,
\[d_{X}(\psi(\lambda)(x_y),x_{\lambda y})\leq d_{Y}(f(\psi(\lambda)(x_{y})),f(x_{\lambda y}))+\epsilon.\]

The triangle inequality yields
\[d_{Y}(f(\psi(\lambda(x_y))),f(x_{\lambda y}))\leq d_{Y}(f(
\psi(\lambda)(x_y)),\lambda f(x_y))+d_{Y}(\lambda f(x_y),f(x_{\lambda y})).\]

Since $\lambda \circ f =f\circ \psi(\lambda)$ up to $\epsilon$, one gets $d_{Y}(f(\psi(\lambda)(x_y)),\lambda f(x_y))\leq \epsilon$. Now  clearly,

\[d_{Y}(\lambda f(x_y),f(x_{\lambda y}))\leq d_{Y}(\lambda f(x_y),\lambda y)+ d_{Y}(\lambda y, f(x_{\lambda y}))\leq 2\epsilon\]

Hence, $d_{X}(\psi(\lambda)(x_y),x_{\lambda y})\leq 4\epsilon$.  Since \[d_{X}(\psi(\lambda)(x_y),x_{\lambda y})=d_{X}(\psi(\lambda)(\tilde{f}(y)),\tilde{f}(\lambda y))\] and $y$ is arbitrary, we see that $\psi(\lambda)\circ \tilde{f} = \tilde{f} \circ \lambda$ up to $4\epsilon$.  A similar argument establishes that for $g\in G_{X}$, $g\circ \tilde{f}= \tilde{f}\circ \theta(g)$ up to $4\epsilon$.
\end{proof}

Analogous to Gromov--Hausdorff convergence, one would like a method of convergence, stronger than Gromov--Hausdorff convergence, that respects groups acting on the spaces. Thus, one formulates the following definition (cf.\ \cite{fukaya1986theory,harvey2014around,harvey2016convergence}).

\begin{definition}
\label{def}
Let $\{X_k\}_{k \in \mathbb{N}}$ be a sequence of compact metric spaces and let $G_k$ be a closed subgroup of Isom$(X_k)$. Let $X$ be a compact metric space and $G$ a closed subgroup of $\Isom(X)$. We say that $(X_k,G_k)$ \textit{equivariantly Gromov--Hausdorff converges} to $(X,G)$, denoted $(X_k,G_k)\rightarrow_{eGH} (X,G)$, if for every $\epsilon>0$, there exists some positive integer $N$ such that for all $k\geq N$ we have an $\epsilon$-equivariant Gromov--Hausdorff approximation $(f_k,\theta_k,\psi_k)$ between $(X_k,G_k)$ and $(X,G)$. 
\end{definition}

We will refer to a triple $(f_k,\theta_k,\psi_k)$ of $\epsilon_k$-equivariant Gromov--Hausdorff approximations appearing in the definition of equivariant Gromov--Hausdorff convergence $(X_k,G_k)\rightarrow_{eGH}(X,G)$, as a \emph{triple that witnesses the convergence} (cf. \cite{harvey2016convergence,harvey2016equivariant}).

The notion of convergence above yields a distance function $d_{eGH}$ on pairs of the form $(X,G)$, where the first factor is a compact space and the second factor a closed subgroup of $\Isom(X)$. In this case, the distance between two such pairs $(X,G_{X})$ and $(Y,G_{Y})$ is defined to be the infimum of all $\epsilon>0$ for which there exists an $\epsilon$-equivariant Gromov--Hausdorff approximation from $(X,G_{X})$ to $(Y,G_{Y})$ and from $(Y,G_{Y})$ to $(X,G_{X})$ \cite{Hausdorff-Fukaya,fukaya1986theory,fukayaandyamaguchiannals}. Proposition~\ref{prop:4e.equiv.GH.approximation} implies the following result.

\begin{prop}
Let $\{X_k\}_{k\in \mathbb{N}}$ be a sequence of compact metrics spaces and $X$ is a compact metric space. Assume for each integer $k$, $G_k$ denotes a closed subgroup of $\Isom(X_k)$ and $G$ is a closed subgroup of $\Isom(X)$. Then the following assertions are equivalent:
\begin{enumerate}
    \item $(X_k,G_k)\rightarrow_{eGH}(X,G)$.
    \item $\lim_{k\rightarrow \infty}d_{eGH}((X_k,G_k),(X,G))=0$.
\end{enumerate}
\end{prop}

\section{Proof of Theorem \ref{T:theorem.B}}
\label{s:proof.technical.thm}
In what follows, we denote by $d_{G_{X}}$ and $d_{G_Y}$ the uniform metrics on $G_{X}$ and $G_Y$, respectively.

\begin{prop}
 Assume $(f,\theta,\psi)$ is an $\epsilon$-equivariant Gromov--Hausdorff approximation between $(X,G_{X})$ and $(Y,G_{Y})$. Then, for any $g_{Y}\in G_{Y}$, there exists a $g_{X}\in G_{X}$ such that $d_{G_{Y}}(g_{Y},\theta(g_{X}))\leq 4\epsilon$.
\end{prop}

\begin{proof}
Let $g\in G_{Y}$ and put $g_{X}=\psi(g)$. Fix $y\in Y$. Since $f$ is an $\epsilon$-isometry, there exists an $x\in X$ such that $d_{X}(f(x),y)\leq \epsilon$. Hence, after applying the triangle inequality twice, one gets
\[
d_{Y}(gy,\theta(g_{X})(y))\leq 2\epsilon +d_{Y}(gf(x),\theta(g_{X})(f(x))).
\]

Thus, it suffices to prove that
\[
d_{Y}(gf(x),\theta(g_{X})(f(x)))\leq 2\epsilon.
\]

To establish this inequality, first observe that
\[
d_{Y}(gf(x),\theta(g_{X})(f(x)))\leq d_{Y}(gf(x),f(g_{X}x))+d_{Y}(f(g_{X}x),\theta(g_{X})(f(x))). 
\]

Recall that $g_{X}=\psi(g)$. Therefore, as $f\circ \psi(g) = g\circ f$ up to $\epsilon$, one gets $d_{Y}(gf(x),f(g_{X}x))\leq \epsilon$. As $f\circ g_{X}=\theta(g_{X})\circ f$ up to $\epsilon$, one gets $d_{Y}(gf(x),\theta(g_{X})(f(x)))\leq 2\epsilon$.
\end{proof}

\begin{prop}
Assume  $(f,\theta,\psi)$ is an $\epsilon$-equivariant Gromov--Hausdorff approximation between $(X,G_{X})$ and $(Y,G_{Y})$. Then, for $g_{X},g_{X}'\in G_{X}$, 
$d_{G_{Y}}(\theta(g_{X}),\theta(g_{X}'))\leq 5\epsilon + d_{G_X}(g_{X},g_{X}')$.
\end{prop}

\begin{proof}
Let $y\in Y$. Choose $x\in X$ such that $d_{Y}(f(x),y)\leq \epsilon$. Hence,
\[
d_{Y}(\theta(g_{X})(y),\theta(g_{X}')(y))\leq \epsilon+ d_{Y}(\theta(g_{X})(f(x)),\theta(g_{X}')(y)).
\]

The triangle inequality along with the fact that $\theta$ takes isometries to isometries, implies that
\[
d_{Y}(\theta(g_{X})(f(x)),\theta(g_{X}')(y))\leq \epsilon+d_{Y}(\theta(g_{X})(f(x)),\theta(g_{X}')(f(x))).
\]

Combining the preceding inequalities, we see that
\[
d_{Y}(\theta(g_{X})(y),\theta(g_{X}')(y))\leq 2\epsilon+d_{Y}(\theta(g_{X})(f(x)),\theta(g_{X}')(f(x))).
\]

To prove the proposition, it suffices to prove that
\[
d_{Y}(\theta(g_{X})(f(x)),\theta(g_{X}')(f(x)))\leq 3\epsilon +d_{G_X}(g_{X},g_{X}').
\]

To prove the preceding inequality, observe first that one clearly has 
\[
d_{Y}(\theta(g_{X})(f(x)), \theta(g_{X}')(f(x)))\leq d_{Y}(\theta(g_{X})(f(x)),f(g_{X}x))+d_{Y}(f(g_{X}x),\theta(g_{X}')(f(x))).
\]

As $\theta(g_{X})\circ f= f\circ g_{X}$ up to $\epsilon$, one gets $d_{Y}(\theta(g_{X})(f(x)),f(g_{X}x))\leq \epsilon$. Thus, it remains to show that $d_{Y}(f(g_{X}x),\theta(g_{X}')(f(x)))\leq 2\epsilon+d_{G_{X}}(g_{X},g_{X}')$. Indeed,
\[
d_{Y}(f(g_{X}x),\theta(g_{X}')(f(x)))\leq d_{Y}(f(g_{X}x),f(g_{X}'x))+d_{Y}(f(g_{X}'x),\theta(g_{X}')(f(x))).
\]

Since $f$ is an $\epsilon$-Gromov--Hausdorff approximation, $d_{Y}(f(g_{X}x),f(g_{X}'x))\leq \epsilon +d_{G_X}(g_{X},g_{X}')$. Since $f\circ g_{X}'=\theta(g_{X}')\circ f$ up to $\epsilon$, the desired inequality then follows.
\end{proof}

The proof of the next proposition is similar to the previous proposition, we therefore omit it.

\begin{prop}
Let $X$ and $Y$ be compact metric spaces. Assume $(f,\theta,\psi)$ is an $\epsilon$-equivariant Gromov--Hausdorff approximation between $(X,G_{X})$ and $(Y,G_{Y})$. Then, for $g_{X},g_{X}'\in G_{X}$, $d_{G_Y}(\theta(g_{X}),\theta(g_{X}'))\geq d_{G_X}(g_{X},g_{X}')-5\epsilon$.
\end{prop}

Now the following corollaries follow from the definition of a Gromov--Hausdorff approximation. Namely, what we have shown is that once we metrize the compact open topology on the isometry groups by the uniform metrics, then the maps $\theta_k$, that are part of a triple $(f_k,\theta_k,\psi_k)$ that is an $\epsilon_k$-equivariant Gromov--Hausdorff approximation demonstrating the equivariant convergence $(X_k,G_k)\rightarrow_{eGH}(X,G)$, are in fact $5\epsilon_k$-Gromov--Hausdorff approximations.

\begin{corollary}
\label{cor:useful}
Let $X_k$ and $X$ be compact metric spaces and let $G_k$ and $G$ be closed subgroups of $\Isom(X_k)$ and $\Isom(X)$ respectively.  If $(f_k,\theta_k,\psi_k)$ is an $\epsilon_k$-equivariant Gromov--Hausdorff approximation from $(X_k,G_k)$ to $(X,G)$ then, under the uniform metric, $\theta_k\colon G_k\rightarrow G$ is an $5\epsilon_k$-Gromov--Hausdorff approximation.
\end{corollary}

\begin{corollary}
Let $\{X_k\}_{k\in \mathbb{N}}$ be a sequence of compact metric spaces and let $G_k$ be a closed subgroup of $\Isom(X_k)$. Assume $X$ is a compact metric space, with corresponding closed subgroup $G$ of $\Isom(X)$ such that $(X_k,G_k)\rightarrow_{eGH}(X,G)$. Then, $G_k\rightarrow_{GH} G$ with respect to the uniform metrics.
\end{corollary}

Notice that what we have proved so far is stronger than the preceding corollary. What we have proved, in fact, is that the maps $\theta_k\colon G_k\rightarrow G$ that demonstrate the equivariant Gromov--Hausdorff convergence are also Gromov--Hausdorff approximations with an \textit{error} that tends to 0.

\begin{remark}While preparing this manuscript, there appeared preprints \cite{cavallucci2023gh,Zamora} with results similar to the preceding corollary. The proof in  \cite{cavallucci2023gh} uses ultralimits, is for pointed spaces, and the statement of their proposition is for geodesic spaces.  The proof in \cite{Zamora} is for pointed spaces and is technically different from ours.
\end{remark}

The fact that the maps $\theta_k\colon G_k\rightarrow G$ that appear in a triple $(f_k,\theta_k,\psi_k)$ that demonstrates the equivariant convergence $(X_k,G_k)\rightarrow_{eGH}(X,G)$ are, once we endow $G_k$ and $G$ with the uniform metrics, $5\epsilon_k$-Gromov--Hausdorff approximations, yields the following result; which will be useful in this section and the next.

\begin{corollary}
\label{nice corollary} Fix $\epsilon>0$ and 
assume $X$ is a compact metric space and $(f,\theta,\psi)$ is an $\epsilon$-equivariant Gromov--Hausdorff approximation from $(X,G)$ to $(X',G')$. Assume $\theta'\colon G\rightarrow G'$ is a map that is $\epsilon$-close to $\theta$ (with respect to the uniform metrics). Then, $(f,\theta',\psi)$ is a $2\epsilon$-equivariant Gromov--Hausdorff approximation. Consequently, $\theta'$ is a $10\epsilon$-Gromov--Hausdorff approximation.
\end{corollary}

\begin{proof}
Let $d_{X}$ denote the metric on $X$ and let $d_{X'}$ denote the metric on $X'$. Let $g\in G$. It suffices to verify the following inequality:
\[
d_{X'}(\theta'(g)f(x),f(gx))\leq 2\epsilon.
\]

Indeed, 
\[
d_{X'}(\theta'(g)f(x),f(gx))\leq d_{X'}(\theta'(g)f(x),\theta(g)f(x))+d_{X'}(\theta(g)f(x),f(gx)).
\]

Since $\theta$ is $\epsilon$-close to $\theta'$, $d_{X'}(\theta'(g)f(x),\theta(g)f(x))\leq \epsilon$. Since $(f,\theta,\psi)$ is an $\epsilon$-equivariant Gromov--Hausdorff convergence, $d_{X'}(\theta(g)f(x),f(gx))\leq \epsilon$.  That $\theta'$ is a $10\epsilon$-Gromov--Hausdorff approximation follows from the first three propositions in this section.
\end{proof}

\subsection*{Proof of Theorem \ref{T:theorem.B}} First, notice that any sequence of isometries on a compact metric space $X$ must be equicontinuous. Therefore, the Arzelà--Ascoli theorem tells us that such a sequence contains a subsequence that converges uniformly to a continuous function on $X$. The limit must be an isometry with respect to the uniform distance. Whence, every closed subgroup $G$ of Isom$(X)$ is compact. Now, let $(f_k,\theta_k,\psi_k)$ be an $\epsilon_k$-equivariant Gromov--Hausdorff approximation that witnesses the convergence $(X_k,G_k)\rightarrow_{eGH}(X,G)$. Then, as we have shown in corollary~\ref{cor:useful}, $\theta_k\colon G_k\rightarrow G$ is a $5\epsilon_k$-Gromov--Hausdorff approximation with respect to the uniform metric. As $G$ is a compact ENR, there exists a topological embedding $\phi \colon G\rightarrow \mathbb{R}^q$ for some $q<\infty$, a compact set $C$ in $\mathbb{R}^q$, for which $\phi(G)$ is contained in the interior of $C$ and for which there exists a retraction $r\colon C\rightarrow \phi(G)$ (cf.\ \cite{Hatcher,ivanov1997gromov}). Now we proceed as in \cite{ivanov1997gromov}.  As each $G_k$ is compact, we may find a finite $5\epsilon_k$-net $Y_k$ in $G_k$. For each natural number $k$, define $\sigma_k$ to be a real-valued function, with domain being $[0,\infty)$ such that $\sigma_k$ is positive on $[0,10\epsilon_k)$ and is zero otherwise. Now define $\phi_k$ to be the following continuous map. For $g_k\in G_k$, set
\[ \phi_k(g_k) = \frac{\sum_{h_k\in Y_k} \sigma_k(|g_k,h_k|)\phi(\theta_k(h_k))}{\sum_{h_k\in Y_k}\sigma_k(|g_k,h_k|)},
\]
where we have denoted the metric on $G_k$ by $|\ ,\ |$. That is, $|g_k,h_k|$ denotes the uniform distance between $g_k$ and $h_k$ in $G_k$. Notice that for any $g_k\in G_k$, there exists a $h_k\in Y_k$ such that $|g_k,h_k|\leq 5\epsilon_k$. Hence, $\sigma_k(|g_k,h_k|)>0$. Therefore the denominator is positive. Since $\phi(G)$ lies in the interior of the compact set $C$, for all sufficiently large $k$, the uniform distance between $\phi^{-1}\circ r\circ \phi_k$ and $\theta_k$ goes to $0$ as $k\to \infty$. Now the result follows from Corollary \ref{nice corollary}.\qed

\section{Proof of Theorem \ref{T:theorem.A} and Corollary \ref{C:corollary.C}.}
\label{s:proof.stability.thm}

In this section, we will first prove a  motivating stability result in the Riemannian manifold setting. We assume all Riemannian manifolds to be without boundary, of finite dimension, compact, connected, and with a lower sectional curvature bound. We establish that under certain nice conditions, equivariant Gromov--Hausdorff convergence implies stability in terms of equivariant cohomology groups. In the manifold setting, we assume all groups act $\emph{smoothly}$. Thus, properness of the action immediately follows. We further note that the motivating result below is essentially an easy consequence of standard facts. However, we include details for the convenience of the reader. For a reference on equivariant cohomology, we refer the reader to the book by Tu \cite{Tu}. 

\begin{prop}[Motivating]
Let $\{X_k\}_{k\in \mathbb{N}}$  be a sequence of connected compact Riemannian manifolds without boundary, with a uniform lower sectional curvature bound, a uniform dimension and uniform upper diameter bound. For each $k$, assume $G_k$ is an $m$-dimensional closed subgroup of $\Isom(X_k)$ acting freely on $X_k$. Suppose further that $(X_k,G_k)\rightarrow_{eGH}(X,G)$, where $X$ is a compact connected Riemannian manifold without boundary, having the same dimension as some (and hence all) $X_k$. If $G$ is an m-dimensional closed subgroup of $\Isom(X)$ acting freely on $X$. Then for all sufficiently large $k$, $H^{*}_{G_k}(X_k)\cong H^{*}_{G}(X)$.
\end{prop}

Note that $H^{*}_{G}(X)$ denotes the equivariant cohomology group of the $G$-space $X$.

\begin{proof}
The projection map $\pi\colon X\rightarrow X/G$ is a principal $G$-bundle. Put $X_G= EG\times_{G} X$. Since $X_{G}\rightarrow X/G$ is a fiber bundle with fiber $EG$, the long exact sequence for homotopy groups tells us that $X_{G}$ and $X/G$ are weakly homotopy equivalent. Similarly, $(X_k)_{G_k}$ is weakly homotopy equivalent to $X_k/G_k$. Now observe that $\dim(X_k/G_k)=\dim(X/G)$ for all $k$. What is more, $X_k/G_k \rightarrow_{GH} X/G$  (with the orbit metric, see \cite{burago2022course}). Therefore, $X_k/G_k$ and $X/G$ are weakly homotopy equivalent for all large enough $k$. 
\end{proof}

The preceding proposition agrees with our intuition. Namely, it is reasonable to expect that once free actions are involved, arguments become much easier to construct. However, clearly not all group actions act freely. Therefore it is desirable to strengthen the previous result. The strengthening is Corollary \ref{C:corollary.C}, which in turn is a consequence of Theorem~\ref{T:theorem.A}.

\begin{remark}
\label{remark}
Before we begin with the proof of Theorem~\ref{T:theorem.A}, a few words are in order. Harvey proved in \cite{harvey2016convergence} that for pointed equivariant convergence, one can take the maps $G_k\rightarrow G$ in the equivariant convergence to be Lie group homomorphisms. In fact, Harvey further showed (\cite{harvey2016convergence}, Proposition 4.1) that once we restrict attention to pointed equivariant convergence, and $n$-dimensional Alexandrov spaces, with curvature bounded below by $K$, the maps \linebreak $G_k\rightarrow G$ in the equivariant Gromov--Hausdorff convergence, provided $G_k$ and $G$ are compact, can be taken to be injective Lie group homomorphisms. In particular, the proof depends on the convergence of balls. Our strategy is to reduce the convergence to a convergence of balls while keeping the maps $G_k\rightarrow G$  that appear in a triple that demonstrates the equivariant convergence fixed. In particular, we need to adjust the convergence and make it local.

\end{remark}

 We now prove Theorem \ref{T:theorem.A} and Corollary \ref{C:corollary.C} together.\newline\
 
\textbf{Convention.} In what follows, we denote by $\epsilon_k$ various positive values that tend to zero as $k\rightarrow \infty$.  In particular, any positive value proportional to $\epsilon_k$ is also denoted by $\epsilon_k$.

\subsection*{Proof of Theorem \ref{T:theorem.A} and Corollary \ref{C:corollary.C}} Let $(f_k,\theta_k,\psi_k)$ be an $\epsilon_k$-equivariant Gromov--Hausdorff approximation that witnesses the equivariant Gromov--Hausdorff convergence\linebreak $(X_k,G_k)\rightarrow_{eGH}(X,G)$, where $\{X_k\}_{k\in \mathbb{N}}$ denotes a sequence of compact Alexandrov spaces with curvature uniformly bounded below, uniform upper diameter bound, $X_k$ and $X$ have the same dimension for all $k$, and $G_k$ and $G$ are closed subgroups of $\Isom(X_k)$ and $\Isom(X)$ respectively, such that the dimension of $G_k$ and $G$ is the same for all $k$.

Following remark \ref{remark}, we need to adjust the convergence and make it local. To that end, fix an arbitrary $x\in X$.  For each integer $k$, choose $x_k\in X_k$ such that $d_{X}(f_k(x_k),x)\leq \epsilon_k$. Define $\tilde{f}_k\colon X_k\rightarrow X$  by 

\[ \tilde{f}_k(x')=\begin{cases} 
      x, &  x'=x_k. \\
      f_k(x'), & \space x'\neq x_k. \\

   \end{cases}
\]

Then, it follows that $\tilde{f}_k$ is an $\epsilon_k$-isometry. Moreover, due to the upper bound on the diameters of the terms of the sequence, we may choose $k$ large enough so that $(\tilde{f}_k,\theta_k,\psi_k)$ demonstrates the equivariant convergence.

For all large enough $k$, the map $\theta_k$ may now be taken to be an injective Lie group homomorphism (see \ref{remark}).  Thus $\theta_k(G_0^k)\subseteq G_0$. Here, $G_0^k$ denotes the identity component of $G_k$ and $G_0$ denotes the identity component of $G$. In fact, we have equality: Indeed, by the rank-nullity theorem, the map $\theta_k$ is a submersion. Define for all such sufficiently large $k$, $\tilde{\theta}_k\colon G_k/G_0^k\rightarrow G/G_0$  by
\[\tilde{\theta}_k(g_kG_0^k)=\theta_k(g_k)G_0\]

Since $\theta_k$ is an injective submersion, it follows that $\tilde{\theta}_k$ is well defined and injective. Hence, the number of components of $G_k$ is at most the number of components of $G$. Equip $G_k$ and $G$ with the uniform metrics, which we denote by $d_{G_k}$ and $d_{G}$ respectively. By Corollary~\ref{nice corollary}, the map $\theta_k\colon G_k\rightarrow G$ is an $\epsilon_k$-isometry.

We will now show that for all large $k$, the number of components between $G_k$ and $G$ is preserved. Regard $G_k$ as a $G_0^k$-space and $G$ as a $G_0$-space. Thereby viewing the cosets as equivalence classes. 
Define a metric $\tilde{d}$ on $G/G_0$, by letting
\[\tilde{d}([g],[g'])=\inf_{g\in [g], g'\in [g']}d_{G}(g,g').\]
In fact, compactness of $G$ ensures that the infimum is a minimum. It follows from \cite{Palais}, that $\tilde{d}$ is a metric on $G/G_0$.  To avoid cumbersome notation, we have regarded the components of $G_k$ and $G$ as equivalence classes.  We will show that for all large enough $k$, the map $\tilde{\theta}_k\colon G_k/G_0^k\rightarrow (G/G_0, \tilde{d})$ is surjective. Let $\epsilon$ be the minimum of the distances between two different cosets in $G/G_0$. Since $G$ is compact, $\epsilon>0$. Choose $k$ large enough so that $0<\epsilon_k<\epsilon.$ If $gG_0\in G/G_0$, then choose $g_k\in G_k$ such that $d_{G}(\theta_k(g_k),g)\leq \epsilon_k$. Then,
\[ \tilde{d}(\tilde{\theta}_k(g_kG_0^k),gG_0)\leq \epsilon_k.\]
Hence, $\tilde{\theta}_k(g_kG_0^k)=gG_0$. Thus for all large $k$, the number of components of $G_k$ and $G$ is preserved. Fix a large enough $k$ so that $G_k$ and $G$ have the same number of components. Then it follows that we must have $\theta_k(G_k)=G$. Thus,  $\theta_k$ is an isomorphism of Lie groups for all large $k$ \cite{Lee}.

Now, define a $G$-action on each $X_k$ as follows. For $g\in G$, define for $x_k\in X_k$,  $g\star x_k=\theta_k^{-1}(g)(x_k)$. Then, $(X_k,G)\rightarrow_{eGH}(X,G)$ (we will generalize this argument below in \ref{general}). Consequently for all large enough $k$, there exists an equivariant homeomorphism $\xi_k\colon X_k\rightarrow X$ (\cite{harvey2016equivariant}, Theorem 3.6). Now for $g_k\in G_k$, and $x_k\in X_k$, 
\[
\xi_k(g_kx_k)=\xi_k(\theta_k^{-1}(\theta_k(g_k))(x_k))=\xi_k(\theta_k(g_k)\star x_k)=\theta_k(g_k)\xi_k(x_k).
\]
Let $\theta \colon EG_k\rightarrow EG$ denote the natural map induced by $\theta_k$. The desired homeomorphism is the map $\Psi\colon EG_k \times_{G_k}X_k\rightarrow EG\times_{G} X$ that is defined by $\Psi([e,x_k])=[\theta(e),\xi(x_k)]$. \qed


\begin{remark}We can generalize the argument indicated above \label{general} as follows: For a sufficiently large $N$, we may use the rank-nullity theorem to see that  $\theta_{N}$ and $\theta_{N+i}^{-1}\theta_{N+i+1}$ for all $i\geq 0$, defines an inductive inverse limit $\tilde{G}=\varprojlim \tilde{G}_i$, where $\tilde{G}_0=G$, $\tilde{G}_i=G_{N+i-1}$ for $i\geq 1$. Then from \cite{nosmallsubgroups}, $\tilde{G}$ is a compact Lie group. For all $k\geq N$, define a $\tilde{G}$ action on $X_k$ and $X$ by projecting in the obvious way. Then, the distance between $(X_k,\tilde{G})$ and $(X,\tilde{G})$ goes to zero. Thus for sufficiently large $k$, $(X_k,\tilde{G})$ and $(X,\tilde{G})$ are equivariantly homeomorphic. Hence, let $\Psi \colon X_k\rightarrow X$ be such a homeomorphism (note, $\Psi$ depends on $k$). For $g_k\in G_k$ and $x_k\in X_k$, it now follows that $\Psi(g_kx_k)=\theta_k(g_k)\Psi(x_k)$.  The ideas of this argument are more general and the ideas apply to different contexts (as an example, see proposition 3.6 in \cite{fukayaandyamaguchiannals}).
\end{remark}

\begin{remark}
\label{effective remark}
Note that in the proof of Theorem \ref{T:theorem.A}, it was crucial that the $G_k$ and $G$ are closed subgroups of the respective isometry groups and thus, naturally define effective actions. We have used this property in a number of places. First, and foremost, the results we used in \cite{harvey2016convergence,harvey2016equivariant} are for effective actions. Second, in order to show that the number of components of the groups $G_k$ and $G$ is eventually preserved we used the fact that the isometry groups can be metrized with the uniform metric. Lastly, to invoke the result in \cite{harvey2016equivariant}, it was important to ensure that the 'new' $G$-action on the $X_k$ is effective. This follows from the fact that the actions of $G_k$ and $G$ are effective.
\end{remark}

Once we assume positive curvature, the diameter restriction in the hypotheses of Theorem \ref{T:theorem.A} is no longer necessary. This is articulated more precisely by the following corollary.


\begin{corollary}
Let $\{X_k\}_{k\in \mathbb{N}}$ and $X$ denote Alexandrov spaces with a positive uniform lower curvature bound. Assume further that the dimension of each $X_k$ agrees with the dimension of $X$. If each $G_k$ is a m-dimensional closed subgroup of $\Isom(X_k)$ and $G$ is an $m$-dimensional closed subgroup of $\Isom(X)$, if $(X_k,G_k)\rightarrow_{eGH}(X,G)$ then for all sufficiently large $k$, $H_{G_k}^{*}(X_k)\cong H_{G}^{*}(X)$.
\end{corollary}



\printbibliography

\end{document}